\documentclass[11pt]{article}
\usepackage{amsfonts,latexsym,rawfonts,amsmath,amssymb,amsthm,a4wide, times, color}
\usepackage{graphicx}
\usepackage[plainpages=false]{hyperref}
\numberwithin{equation}{section}
\usepackage{hyperref}
\usepackage[titletoc,title]{appendix}

 \usepackage[all]{xy}
 \usepackage{eufrak}
 \usepackage{makeidx}
 \usepackage{graphicx,psfrag}

\numberwithin{equation}{section}

\newcommand{\beq}{\begin{equation}}
\newcommand{\eeq}{\end{equation}}
\newcommand{\beqs}{\begin{eqnarray*}}
\newcommand{\eeqs}{\end{eqnarray*}}
\newcommand{\beqn}{\begin{eqnarray}}
\newcommand{\eeqn}{\end{eqnarray}}
\newcommand{\beqa}{\begin{array}}
\newcommand{\eeqa}{\end{array}}

\def\x{{\mathbf{x}}}

\def\n{{\mathbf{n}}}

\newtheorem{prop}{Proposition}[section]
\newtheorem{theo}[prop]{Theorem}
\newtheorem{lem}[prop]{Lemma}

\newtheorem{cor}[prop]{Corollary}

\title{ Self-shrinkers with bounded HA}
\author{Zhen Wang 
}

\begin{document}
\bibliographystyle{plain}


\maketitle

\begin{abstract}
    We study integral and pointwise bounds on the second fundamental form of properly immersed self-shrinkers with bounded $HA$. As applications, we discuss gap and compactness results for self-shrinkers. 
\end{abstract}

\tableofcontents

\section{Introduction}
A hypersurface $\Sigma\hookrightarrow\mathbb{R}^{n+1}$ is said to be a self-shrinker if it is the time $t=-1$ slice of a mean curvature flow moving by rescalings with $\Sigma_t=\sqrt{-t}\Sigma$, or equivalently if it satisfies the equation
$$H=\frac{\langle\x,\n\rangle}{2},$$
where $\n$ and $H$ denote the unit normal vector and the mean curvature, respectively. Self-shrinkers play an important role in the study of mean curvature flow, not least because they are models for type-I singularities of the flow by Huisken \cite{[Hu1],[Hu2]}.

It is interesting to compare the $HA$ tensor in mean curvature flow with the Ricci curvature in Ricci flow since they describe the corresponding metric evolution, respectively. In \cite{[CW],[KMW]} Chen-Wang and Kotschwar-Munteanu-Wang showed the Ricci curvature blows up at the rate of type-I at the first finite singularity. In \cite{[LS3]} Sesum proved the type-I blowup of mean curvature at the finite type-I singularity. 

In \cite{[LW2]} Li-Wang studied the flow with type-I mean curvature and confirmed the multiplicity-one conjecture in this case. The present paper follows the method of \cite{[MW]} and can be seen as an attempt to understand more about the asymptotic behaviour of self-shrinkers in terms of $HA$.

Let $f=|x|^2/4$. By integral estimates and the Moser iteration we get the following pointwise growth estimate of the second fundamental form.
\begin{theo}(Theorem \ref{theo:growth rate})
    Let $\x:\Sigma^n\to\mathbb{R}^{n+1}$ be a properly immersed self-shrinker with $\sup_{\Sigma}|HA|\leq K$. Then for any $p>\max\{n,4\}$ there exist positive constants $C=C(n,p,K,\int_{\Sigma}e^{-f},\int_{B(0,r_0)\cap\Sigma}|A|^p)$ where $r_0=c(n,p)(1+K)$ and $a=a(n,p,K)$ such that
    \beqs
    |A|(x)\leq C(|x|+1)^a,\quad \forall\,x\in\Sigma,
    \eeqs
    i.e., the second fundamental form grows at most polynomially in the distance.
\end{theo}

Based on the polynomial growth of volume and second fundamental form, we find that a self-shrinker with sufficiently small $|HA|$ must be a hyperplane.
\begin{theo}(Corollary \ref{cor:hyperplane gap})
Let $\x:\Sigma^n\to\mathbb{R}^{n+1}$ be a smooth properly embedded self-shrinker. There exists a constant $\varepsilon_n=\frac1{\sqrt{n}(n+5)^4}$ such that if $\sup_{\Sigma}|HA|\leq\varepsilon_n$ then $\Sigma$ is a hyperplane through 0.
\end{theo}

By similar argument we find that a local energy bound implies a global energy bound.

\begin{theo}(Proposition \ref{prop:energy bound})
    Let $\x:\Sigma^n\to\mathbb{R}^{n+1}$ be a properly immersed self-shrinker with $n\geq4$ and $\sup_{\Sigma}|HA|\leq K$. Then there exists a $r_1=c_n\sqrt{K}$ such that if 
    $\int_{B(0,r_1)\cap\Sigma}|A|^n\leq E$,
    then for any $r>0$ we have
    $\int_{B(0,r)\cap\Sigma}|A|^n\leq 3Ee^{r^2/4}$.
\end{theo}

By virtue of the energy estimate above and the $\epsilon-$regularity from Li-Wang \cite{[LW]} we find that the space of properly embedded self-shrinkers with uniformly bounded entropy, uniformly bounded $|HA|$ and uniformly bounded local energy is compact. 

\begin{theo}(Theorem \ref{theo:compactness})
    Let $\{\Sigma^n_i\}$ be a sequence of properly embedded self-shrinkers with $n\geq4$ normalized by $\int_{\Sigma_i}e^{-f}\leq (4\pi)^{n/2}$. Assume that $\sup_{i}\sup_{\Sigma_i}|HA|\leq K$ and $\sup_{i}\int_{B(0,r_1)\cap\Sigma_i}|A|^n<\infty$ where $r_1=c_n\sqrt{K}$ is the positive constant in Proposition \ref{prop:energy bound}. Then a subsequence of $\{\Sigma_i\}$ converges smoothly to a smooth properly embedded self-shrinker $\Sigma_{\infty}$.
\end{theo}

The organization of this paper is as follows. In Sect.2 we recall some results on self-shrinkers and differential equations. In Sect.3 we develop $L^p$ estimate of $A$ and derive pointwise estimate by standard Moser iteration. In Sect.4 we obtain the gap theorem using weighted integral estimate and get the convergence result by $\epsilon-$regularity.

{\bf Acknowledgements}: The author would like to thank his advisor H. Z. Li for suggesting this problem. Z. Wang is very grateful to I. Khan and H. B. Fang for their insightful discussions.

\section{Preliminaries}
Let $\x:\Sigma^n\to\mathbb{R}^{n+1}$ be a hypersurface without boundary. $\Sigma$ is called a self-shrinker if it satisfies
\beqs
H=\frac{\langle\x,\n\rangle}2.
\eeqs
Throughout this paper, we set the potential function
$$f:=\frac{|x|^2}4.$$
Some related equations are listed below for later calculations. See the proof of Lemma 3.20 and Theorem 5.2 of \cite{[CM2]} for details.
\begin{lem}\label{lem:eqns}
    On a self-shrinker we have
\beqn
&&\Delta f-|\nabla f|^2=\frac{n}2-f,\label{eqn0}\\
&&HA+\nabla^2{f}=\frac12g,\label{eqn1}\\
&&H^2+\Delta{f}=\frac n2,\label{eqn2}\\
&&|\nabla{f}|^2+H^2=f,\label{eqn3}\\
&&\nabla^2H-\nabla f\cdot\nabla A=\frac12A-HA^2,\label{eqn4}\\
&&\Delta A-\nabla f\cdot\nabla A=(\frac12-|A|^2)A,\label{eqn5}\\
&&\frac12 \Big(\Delta|A|^2-\nabla f\cdot\nabla|A|^2\Big)=|\nabla A|^2+(\frac12-|A|^2)|A|^2,\label{eqn6}\\
&&\Delta H-\nabla f\cdot\nabla H=(\frac12-|A|^2)H,\label{eqn7}\\
&&\frac12\Big(\Delta H^2-\nabla f\cdot\nabla H^2\Big)=|\nabla H|^2+(\frac12-|A|^2)H^2\label{eqn8}.
\eeqn
\end{lem}

By Corollary 2.8 of \cite{[CM2]}, we know
\begin{lem}\label{lem:hyperplane}(Corollary 2.8 of \cite{[CM2]})
If $\Sigma$ is a self-shrinker and $H\equiv0$, then $\Sigma$ is a minimal cone. In particular, if $\Sigma$ is also smooth and embedded, then it is a hyperplane through 0.
\end{lem}

From \cite{[CZ]} one sees the equivalence of weighted volume finiteness, polynomial volume growth and properness of an immersed self-shrinker in Euclidean space.

\begin{lem}\label{lem:vol est}(Theorem 1.1 of \cite{[CZ]})
Let $\Sigma^n$ be a complete noncompact properly immersed self-shrinker in Eucildean space $\mathbb{R}^{n+1}$. Then $\Sigma$ has finite weighted volume 
$$\mbox{Vol}_{f}(\Sigma)=\int_{\Sigma}e^{-f}dv<+\infty$$
and
$$\mbox{Vol}(B(0,r)\cap\Sigma)\leq Cr^n,\quad\forall\,r>0,$$
where $C$ is a positive constant depending only on $\int_{\Sigma}e^{-f}dv$.
\end{lem}

Provided bounded mean curvature we can also get volume ratio lower bound. See Lemma 3.5 in Li-Wang \cite{[LW]}.
\begin{lem}\label{lem:volume ratio lower bound}(Lemma 3.5 of \cite{[LW]})
    Let $\Sigma^n\hookrightarrow\mathbb{R}^{n+1}$ be a properly immersed hypersurface in $B(x_0,r_0)$ with $x_0\in\Sigma$ and $\sup_{\Sigma}|H|\leq \Lambda$. Then for any $s\in(0,r_0)$ we have 
    $$\frac{\mbox{Vol}_{\Sigma}(B(x_0,s)\cap\Sigma)}{\omega_n s^n}\leq e^{\Lambda r_0}\frac{\mbox{Vol}_{\Sigma}(B(x_0,r_0)\cap\Sigma)}{\omega_n {r_0}^n}.$$
    In particular,
    $$\mbox{Vol}(B(x_0,r)\cap\Sigma)\geq e^{-\Lambda r}\omega_n r^n,\quad\forall\,r\in(0,r_0].$$
    \end{lem}

In order to obtain the growth rate of second fundamental form from the $L^p$ estimate we will apply the standard Moser iteration. Recall the Michael-Simon inequality which needs mean curvature. Here we present a precise estimate of elliptic case derived from \cite{[LS2]} and \cite{[Li]}. 
\begin{lem}[\textbf{Moser iteration}]\label{lem:Moser iteration}
    Let $\Sigma^n\hookrightarrow\mathbb{R}^{n+1}$ be a hypersurface without boundary.
    Consider the differential inequality 
    $$-\Delta u\leq \varphi u,\quad u\geq0.$$
    Fix $x_0\in\Sigma$ and denote $D_r:=B(x_0,r)\cap\Sigma$. Then for any $r>0$, $q>\frac{n}{2}$ and $\beta\geq2$ there exists a positive constant $C=C(n,q,\beta)$ such that 
    \beqs    
    \|u\|_{L^{\infty}(D_{r/2})}
    \leq Cr^{-\frac{2n^2}{\beta}}\Big(\|\varphi\|_{L^q(D_r)}^{\frac{2q}{2q-n}}+\|H\|_{L^{n+2}(D_r)}^{n+2}\Big)^{\frac{n^2}{\beta}}
    \|u\|_{L^{\beta}(D_r)}.
    \eeqs
\end{lem}

Finally we recall some technical results on interior estimates and compactness of immersed hypersurfaces in $\mathbb{R}^{n+1}$. Note that if $\Sigma$ is a self-shrinker then $\{\Sigma_t=\sqrt{-t}\Sigma,\,-\frac32\leq t\leq-\frac12\}$ is a mean curvature flow, i.e., $\partial_t\x=-H\n$. We obtain some kind of $\epsilon-$regularity from Corollary 3.11 and Theorem 3.7 in Li-Wang \cite{[LW]} and the interior estimates of Ecker and Huisken in \cite{[Ma]}.

\begin{lem}[\textbf{$\epsilon$-regularity}]\label{lem:epsilon-regularity}
There exist constants $\epsilon=\epsilon(n)>0$, $\delta=\delta(n)>0$, $\eta=\eta(n)>0$ and $\{D_k(n,\theta)\}_{k\geq1}$ satisfying the following properties.
Let $\Sigma^n\hookrightarrow\mathbb{R}^{n+1}$ be a properly immersed self-shrinker satisfying $\sup_{\Sigma}|H|\leq\Lambda$. 
If 
$$\int_{B(x_0,r)\cap\Sigma}|A|^n\leq \epsilon$$
for some $x_0\in\mathbb{R}^{n+1}$ and some $0<r\leq\frac1\Lambda$, then we have
$$\sup_{B(x_0,r/2)\cap\Sigma}|A|\leq \frac1r;$$

$$\sup_{B(x_0,r/32)\cap\Sigma_t}|A|\leq \frac2{\delta r},\quad\forall\,t+1\in[-\frac{\eta r^2}{16},\frac{\eta r^2}{16}]\cap[-\frac12,\frac12];$$

$$\sup_{B(x_0,\sqrt{\theta}R)\cap\Sigma}|\nabla^k A|\leq \frac{2D_k(n,\theta)}{\delta r},\quad\forall\,\theta\in(0,\frac12],\quad\forall\,k\geq1,$$
where $R:=\min\{\frac{r}{32},\frac{\sqrt{n\eta}}2r, \sqrt{2n}\}$.
\end{lem}

The following compactness result of mean curvature flow is well-known. See \cite{[CY]} for a detailed proof. 
\begin{lem}[\textbf{Compactness of mean curvature flow}]\label{lem:compactness of MCF}
Let $\{(\Sigma^n_i,\x_i(t)),-1<t<1\}$ be a sequence of mean curvature flow properly immersed in $B(0,R)\subset\mathbb{R}^{n+1}$. Suppose that
$$\sup_{B(0,R)\cap\Sigma_{i,t}}|A|(\cdot,t)\leq \Lambda,\quad\forall\,t\in(-1,1)$$
for some $\Lambda>0$. Then a subsequence of $\{(B(0,R)\cap\Sigma_{i,t}),-1<t<1\}$ converges in smooth topology to a smooth mean curvature flow $\{\Sigma_{\infty,t},-1<t<1\}$ in $B(0,R)$. 
\end{lem}

\section{$L^p$ estimate and growth rate}

Throughout the section we set 
\beqs
\sup_{\Sigma}|HA|\leq K.
\eeqs

\begin{prop}[\textbf{$L^p$ estimate}]\label{prop:Lp est}
Let $\x:\Sigma^n\to\mathbb{R}^{n+1}$ be a properly immersed self-shrinker with $\sup_{\Sigma}|HA|\leq K$. Then for any $p\geq4$ there exists positive constants $a=a(n,p,K)$ and $C=C(n,p,K,\int_{\Sigma}e^{-f},\int_{B(0,r_0)\cap\Sigma}|A|^p)$ where $r_0=c(n,p)(1+K)$ such that 
\beqs
\int_{\Sigma}|A|^p(|x|^2+1)^{-a}\leq C.
\eeqs
Moreover, for any $x\in\Sigma$,
\beqs
\int_{B(x,1)\cap\Sigma}|A|^p\leq C(|x|^2+1)^a.
\eeqs
\end{prop}
\begin{proof}
We always use $c$ to denote a nonegative constant depending only on $n$ and $p$.
For $a>0$ and $p>1$, integrating by parts we have
\beqn
&&a\int_{\Sigma}|\nabla f|^2(f+1)^{-a-1}|A|^p\phi
=-\int_{\Sigma}\nabla f\cdot\nabla(f+1)^{-a}|A|^p\phi\label{est6}\\
&=&\int_{\Sigma}\Delta f(f+1)^{-a}|A|^p\phi
+\int_{\Sigma}\nabla f\cdot\nabla|A|^p\cdot(f+1)^{-a}\phi\nonumber\\
&&+\int_{\Sigma}\nabla f\cdot\nabla\phi\cdot(f+1)^{-a}|A|^p.\nonumber
\eeqn
Note that
\beqs
a|\nabla f|^2(f+1)^{-a-1}-\Delta f(f+1)^{-a}
=\Big(a\frac{f-H^2}{f+1}+H^2-\frac{n}2\Big)(f+1)^{-a}.
\eeqs
Since $H$ is bounded by $n^{\frac14}K^{\frac12}$ there exists a positive constant 
\beqn
r_0=\Big(\frac8n(1+n^{1/2}K)a\Big)^{1/2}\label{r_0}
\eeqn
such that 
\beqs
a\frac{f-H^2}{f+1}+H^2-\frac{n}2\geq a-n, \quad\forall\;|x|\geq r_0.
\eeqs
Let $\phi(x):=\eta(|x|^2)$ be a cutoff where $\eta:[0,\infty)\to\mathbb{R}$ is a nonnegative decreasing Lipschitz function. Thus
\beqs
\nabla f\cdot\nabla\phi=4\eta'|\nabla f|^2\leq0.
\eeqs
Moreover, for any $r>0$ and $0<\delta<1$ fixed, let
$$\eta\equiv1\quad\mbox{on}\;[0,r^2]\;;\quad \eta\equiv0\quad\mbox{on}\;[4r^2,\infty)\;;\quad |\eta'|\leq\frac{1}{3\delta r^2}\eta^{1-\delta},$$
which implies
$$\phi^{-1}|\nabla\phi|^2\leq \frac{16}{9\delta^2r^2}\phi^{1-2\delta}.$$
Back to (\ref{est6}),
\beqn
&&(a-n)\int_{\Sigma}|A|^p(f+1)^{-a}\phi\label{est7}\\
&\leq& \int_{\Sigma}\nabla f\cdot\nabla|A|^p\cdot(f+1)^{-a}\phi+\int_{\Sigma}4\eta'|\nabla f|^2(f+1)^{-a}|A|^p+C_1,\nonumber
\eeqn
where
\beqn
C_1&=&C_1(p)\label{C_1}\\
&:=&\int_{\{|x|\leq r_0\}\cap\Sigma}\Big(-a\frac{f-H^2}{f+1}-H^2+\frac{n}2+a-n\Big)(f+1)^{-a}|A|^p\phi\nonumber\\
&\leq& \int_{\{|x|\leq r_0\}\cap\Sigma}\Big(a\frac{1+H^2}{f+1}-\frac{n}2\Big)(f+1)^{-a}|A|^p\nonumber\\
&\leq& ca(1+K)\int_{\{|x|\leq r_0\}\cap\Sigma}|A|^p(f+1)^{-a-1}<\infty.\nonumber
\eeqn
Using (\ref{eqn4}) and integrating by parts we get
\beqn
&&\int_{\Sigma}\nabla f\cdot\nabla|A|^p\cdot(f+1)^{-a}\phi
=p\int_{\Sigma}\frac12\langle\x,\nabla A\rangle A|A|^{p-2}(f+1)^{-a}\phi\label{est8}\\
&\leq& p\int_{\Sigma}\Big(\nabla^2 H+HA^2-\frac12A\Big)A|A|^{p-2}(f+1)^{-a}\phi\nonumber\\
&\leq& p\int_{\Sigma}\nabla^2 H\cdot A\cdot|A|^{p-2}(f+1)^{-a}\phi
+(pK-\frac{p}2)\int_{\Sigma}|A|^p(f+1)^{-a}\phi\nonumber\\
&\leq& p(p-1)\int_{\Sigma}|\nabla H||\nabla A||A|^{p-2}(f+1)^{-a}\phi+ap\int_{\Sigma}|\nabla H||\nabla f||A|^{p-1}(f+1)^{-a-1}\phi\nonumber\\
&&+p\int_{\Sigma}|\nabla H||\nabla\phi||A|^{p-1}(f+1)^{-a}+(pK-\frac{p}2)\int_{\Sigma}|A|^p(f+1)^{-a}\phi.\nonumber
\eeqn
Using the Caucuy-Schwarz inequality and the Young's inequality, we estiamte the right hand side of (\ref{est8}) above as follows:
\beqs
&&p(p-1)\int_{\Sigma}|\nabla H||\nabla A||A|^{p-2}(f+1)^{-a}\phi\\
&\leq& cK^{-1}\int_{\Sigma}|\nabla H|^2|A|^{p}(f+1)^{-a}\phi
+cK\int_{\Sigma}|\nabla A|^2|A|^{p-4}(f+1)^{-a}\phi,
\eeqs

\beqs
&&ap\int_{\Sigma}|\nabla H||\nabla f||A|^{p-1}(f+1)^{-a-1}\phi\\
&\leq& cK^{-1}\int_{\Sigma}|\nabla H|^2|A|^{p}(f+1)^{-a}\phi
+ca^2K\int_{\Sigma}|\nabla f|^2|A|^{p-2}(f+1)^{-a-2}\phi\\
&\leq& cK^{-1}\int_{\Sigma}|\nabla H|^2|A|^{p}(f+1)^{-a}\phi
+ca^2K\int_{\Sigma}|A|^{p-2}(f+1)^{-a-1}\phi,
\eeqs

\beqs
&&ca^2\int_{\Sigma}|A|^{p-2}(f+1)^{-a-1}\phi
=\int_{\Sigma}|A|^{p-2}(f+1)^{-\frac{p-2}{p}a}\cdot ca^2(f+1)^{-\frac2pa-1}\cdot\phi\\
&\leq& c\int_{\Sigma}|A|^{p}(f+1)^{-a}\phi
+ca^p\int_{\Sigma}(f+1)^{-a-\frac{p}2}\phi,
\eeqs

\beqs
&&p\int_{\Sigma}|\nabla H||\nabla\phi||A|^{p-1}(f+1)^{-a}\\
&\leq& cK^{-1}\int_{\Sigma}|\nabla H|^2|A|^{p}(f+1)^{-a}\phi
+cK\int_{\Sigma}\phi^{-1}|\nabla\phi|^2|A|^{p-2}(f+1)^{-a}\\
&\leq& cK^{-1}\int_{\Sigma}|\nabla H|^2|A|^{p}(f+1)^{-a}\phi
+\frac{cK}{\delta^2r^2}\int_{\Sigma}\phi^{1-2\delta}|A|^{p-2}(f+1)^{-a},
\eeqs

\beqs
&&\frac{c}{\delta^2r^2}\int_{\Sigma}\phi^{1-2\delta}|A|^{p-1}(f+1)^{-a}
=\int_{\Sigma}|A|^{p-1}\phi^{\frac{p-1}{p}}\cdot\frac{c}{\delta^2r^2}\phi^{\frac{1-2p\delta}{p}}\cdot(f+1)^{-a}\\
&\leq& c\int_{\Sigma}|A|^p(f+1)^{-a}\phi
+\frac{c}{\delta^{2p}r^{2p}}\int_{\Sigma}\phi^{1-2p\delta}(f+1)^{-a}.
\eeqs
Let $\delta=\frac1{4p}$ so that $1-2p\delta=\frac12>0$. Then plugging the estimates above into (\ref{est8}) yidles
\beqn
&&\int_{\Sigma}\nabla f\cdot\nabla|A|^p\cdot(f+1)^{-a}\phi\label{est9}\\
&\leq& cK^{-1}\int_{\Sigma}|\nabla H|^2|A|^{p}(f+1)^{-a}\phi
+cK\int_{\Sigma}|\nabla A|^2|A|^{p-4}(f+1)^{-a}\phi\nonumber\\
&&+cK\int_{\Sigma}|A|^{p}(f+1)^{-a}\phi\nonumber
+c(a^p+r^{-2p})K\int_{\Sigma}(f+1)^{-a}.\nonumber
\eeqn
Furthermore, using (\ref{eqn8}) we have
\beqs
&&\int_{\Sigma}|\nabla H|^2|A|^{p}(f+1)^{-a}\phi\\
&=&\int_{\Sigma}\Big(\frac12\Delta H^2-\frac12\nabla f\cdot\nabla H^2+(|A|^2-\frac12)H^2\Big)|A|^{p}(f+1)^{-a}\phi\\
&\leq& -\frac{p}2\int_{\Sigma}\nabla H^2\cdot\nabla A\cdot A\cdot |A|^{p-2}(f+1)^{-a}\phi
+\frac{a}2\int_{\Sigma}\nabla H^2\cdot\nabla f\cdot |A|^{p}(f+1)^{-a-1}\phi\\
&&-\frac12\int_{\Sigma}\nabla H^2\cdot\nabla\phi\cdot|A|^{p}(f+1)^{-a}
-\frac12\int_{\Sigma}\nabla f\cdot\nabla H^2\cdot|A|^{p}(f+1)^{-a}\phi\\
&&+K^2\int_{\Sigma}|A|^p(f+1)^{-a}\phi.
\eeqs
Then,
\beqs
&&\int_{\Sigma}|\nabla H|^2|A|^{p}(f+1)^{-a}\phi\\
&\leq& pK\int_{\Sigma}|\nabla H||\nabla A||A|^{p-2}(f+1)^{-a}\phi
+(a+1)K\int_{\Sigma}|\nabla H||\nabla f||A|^{p-1}(f+1)^{-a}\phi\\
&& +K\int_{\Sigma}|\nabla H||\nabla\phi||A|^{p-1}(f+1)^{-a}
+K^2\int_{\Sigma}|A|^p(f+1)^{-a}\phi\\
&\leq& \frac14\int_{\Sigma}|\nabla H|^2|A|^p(f+1)^{-a}\phi
+p^2K^2\int_{\Sigma}|\nabla A|^2|A|^{p-4}(f+1)^{-a}\phi\\
&&+\frac14\int_{\Sigma}|\nabla H|^2|A|^p(f+1)^{-a}\phi
+(a+1)^2K^2\int_{\Sigma}|\nabla f|^2|A|^{p-2}(f+1)^{-a}\phi\\
&&+\frac14\int_{\Sigma}|\nabla H|^2|A|^p(f+1)^{-a}\phi
+K^2\int_{\Sigma}\phi^{-1}|\nabla\phi|^2|A|^{p-2}(f+1)^{-a}\\
&&+K^2\int_{\Sigma}|A|^p(f+1)^{-a}\phi.
\eeqs
Thus,
\beqn
&&\int_{\Sigma}|\nabla H|^2|A|^p(f+1)^{-a}\phi\label{est10}\\
&\leq& 4p^2K^2\int_{\Sigma}|\nabla A|^2|A|^{p-4}(f+1)^{-a}\phi
+4(a+1)^2K^2\int_{\Sigma}|A|^{p-2}(f+1)^{-a+1}\phi\nonumber\\
&&+cr^{-2}K^2\int_{\Sigma}\phi^{1-2\delta}|A|^{p-2}(f+1)^{-a}
+4K^2\int_{\Sigma}|A|^p(f+1)^{-a}\phi\nonumber\\
&\leq& cK^2\int_{\Sigma}|\nabla A|^2|A|^{p-4}(f+1)^{-a}\phi
+4K^2\int_{\Sigma}|A|^p(f+1)^{-a}\phi\nonumber\\
&&+cK^2\int_{\Sigma}|A|^p(f+1)^{-a}\phi
+c(a+1)^{p}K^2\int_{\Sigma}(f+1)^{-a+\frac{p}2}\nonumber\\
&&+cK^2\int_{\Sigma}|A|^p(f+1)^{-a}\phi
+cr^{-p}K^2\int_{\Sigma}(f+1)^{-a}\nonumber\\
&\leq& cK^2\int_{\Sigma}|\nabla A|^2|A|^{p-4}(f+1)^{-a}\phi
+cK^2\int_{\Sigma}|A|^p(f+1)^{-a}\phi\nonumber\\
&&+c\Big((a+1)^{p}+r^{-p}\Big)K^2\int_{\Sigma}(f+1)^{-a+\frac{p}2}.\nonumber
\eeqn
On the other hand, for any $p\geq4$ by (\ref{eqn6}) we have
\beqs
&&\int_{\Sigma}|\nabla A|^2|A|^{p-4}(f+1)^{-a}\phi\\
&=& \int_{\Sigma}\Big(\frac12\Delta|A|^2-\frac12\nabla f\cdot\nabla|A|^2+(|A|^2-\frac12)|A|^2\Big)|A|^{p-4}(f+1)^{-a}\phi\\
&\leq& \frac{a}2\int_{\Sigma}\nabla|A|^2\cdot\nabla f\cdot|A|^{p-4}(f+1)^{-a-1}\phi
-\frac12\int_{\Sigma}\nabla|A|^2\cdot\nabla\phi\cdot|A|^{p-4}(f+1)^{-a}\\
&&-\frac12\int_{\Sigma}\nabla f\cdot\nabla|A|^2\cdot|A|^{p-4}(f+1)^{-a}\phi
+\int_{\Sigma}|A|^p(f+1)^{-a}\phi.
\eeqs
Then,
\beqs
&&\int_{\Sigma}|\nabla A|^2|A|^{p-4}(f+1)^{-a}\phi\\
&\leq& (a+1)\int_{\Sigma}|\nabla A||\nabla f||A|^{p-3}(f+1)^{-a}\phi
+\int_{\Sigma}|\nabla A||\nabla\phi||A|^{p-3}(f+1)^{-a}\\
&&+\int_{\Sigma}|A|^p(f+1)^{-a}\phi\\
&\leq& \frac14\int_{\Sigma}|\nabla A|^2|A|^{p-4}(f+1)^{-a}\phi
+(a+1)^2\int_{\Sigma}|\nabla f|^2|A|^{p-2}(f+1)^{-a}\phi\\
&&+\frac14\int_{\Sigma}|\nabla A|^2|A|^{p-4}(f+1)^{-a}\phi
+\int_{\Sigma}\phi^{-1}|\nabla\phi|^2|A|^{p-2}(f+1)^{-a}\\
&&+\int_{\Sigma}|A|^p(f+1)^{-a}\phi,
\eeqs

\beqn
&&\int_{\Sigma}|\nabla A|^2|A|^{p-4}(f+1)^{-a}\phi\label{est11}\\
&\leq& 2(a+1)^2\int_{\Sigma}|A|^{p-2}(f+1)^{-a+1}\phi
+cr^{-2}\int_{\Sigma}\phi^{1-2\delta}|A|^{p-2}(f+1)^{-a}\nonumber\\
&&+\int_{\Sigma}|A|^p(f+1)^{-a}\phi\nonumber\\
&=& \int_{\Sigma}|A|^{p-2}(f+1)^{-\frac{p-2}{p}a}\cdot2(a+1)^2(\frac{|x|^2}{4}+1)^{-\frac2pa+1}\cdot\phi\nonumber\\
&&+\int_{\Sigma}|A|^{p-2}\phi^{\frac{p-2}{p}}\cdot cr^{-2}\phi^{\frac{2(1-p\delta)}{p}}\cdot(f+1)^{-a}
+\int_{\Sigma}|A|^p(f+1)^{-a}\phi\nonumber\\
&\leq& c\int_{\Sigma}|A|^p(f+1)^{-a}\phi
+c(a+1)^p\int_{\Sigma}(f+1)^{-a+\frac{p}2}\phi\nonumber\\
&&+c\int_{\Sigma}|A|^p(f+1)^{-a}\phi
+cr^{-p}\int_{\Sigma}(f+1)^{-a}\phi^{1-p\delta}\nonumber\\
&&+\int_{\Sigma}|A|^p(f+1)^{-a}\phi\nonumber\\
&\leq& c\int_{\Sigma}|A|^p(f+1)^{-a}\phi+c\Big((a+1)^p+r^{-p}\Big)\int_{\Sigma}(f+1)^{-a+\frac{p}2}.\nonumber
\eeqn

Combining (\ref{est9}), (\ref{est10}) and (\ref{est11}) we conclude
\beqs
&&\int_{\Sigma}\nabla f\cdot\nabla|A|^p\cdot(f+1)^{-a}\phi\\
&\leq& cK\int_{\Sigma}|\nabla A|^2|A|^{p-4}(f+1)^{-a}\phi
+cK\int_{\Sigma}|A|^{p}(f+1)^{-a}\phi\\
&&+c\Big(r^{-2p}+r^{-p}+(a+1)^{p}\Big)K\int_{\Sigma}(f+1)^{-a+\frac{p}2}\\
&\leq& cK\int_{\Sigma}|A|^{p}(f+1)^{-a}\phi
+c\Big(r^{-2p}+r^{-p}+(a+1)^{p}\Big)K\int_{\Sigma}(f+1)^{-a+\frac{p}2},
\eeqs 
which together with (\ref{est7}) implies
\beqs
(a-n-cK)\int_{\Sigma}|A|^p(f+1)^{-a}\phi
\leq c\Big(r^{-2p}+r^{-p}+(a+1)^p\Big)K\int_{\Sigma}(f+1)^{-a+\frac{p}2}+C_1.
\eeqs
Recall the volume estimate in Lemma \ref{lem:vol est}. Take 
$$a=n+p+cK+1$$ 
so that the right hand side above makes sense. Recall the settings (\ref{r_0}) and (\ref{C_1}) one sees
$$r_0\leq c(1+K),$$
$$C_1(p)\leq c(1+K)^2\int_{\{|x|\leq r_0\}\cap\Sigma}|A|^p.$$
Letting $r\to\infty$ yields 
\beqs
\int_{\Sigma}|A|^p(f+1)^{-a}\leq C(n,p,K)\int_{\Sigma}(f+1)^{-a+\frac{p}2}+c(1+K)^2\int_{\{|x|\leq r_0\}\cap\Sigma}|A|^p<\infty.
\eeqs
In particular, we restrict the integration on $B(x_0,1)\cap\Sigma$ for any $x_0\in\Sigma$, then 
\beqs
\Big(\frac{(|x_0|+1)^2}4+1\Big)^{-a}\int_{B(x_0,1)\cap\Sigma}|A|^p\leq\int_{B(x_0,1)\cap\Sigma}|A|^p(f+1)^{-a}\phi\leq C,
\eeqs
i.e.,
\beqs
\int_{B(x_0,1)\cap\Sigma}|A|^p\leq C(|x_0|^2+1)^a,
\eeqs
where 
$$a=a(n,p,K),\quad C=C(n,p,K,\int_{\Sigma}e^{-f},\int_{B(0,r_0)\cap\Sigma}|A|^p).$$
\end{proof}

\begin{theo}[\textbf{growth rate}]\label{theo:growth rate}
    Let $\x:\Sigma^n\to\mathbb{R}^{n+1}$ be a properly immersed self-shrinker with $\sup_{\Sigma}|HA|\leq K$. Then for any $p>\max\{n,4\}$ there exist positive constants $C=C(n,p,K,\int_{\Sigma}e^{-f},\int_{B(0,r_0)\cap\Sigma}|A|^p)$ where $r_0=c(n,p)(1+K)$ and $a=a(n,p,K)$ such that
    \beqs
    |A|(x)\leq C(|x|+1)^a,\quad \forall\,x\in\Sigma,
    \eeqs
    i.e., the second fundamental form grows at most polynomially in the distance.
\end{theo}
\begin{proof}
Fix $q>\max\{n/2,2\}$.
From (\ref{eqn6}) we know
\beqs
\Delta|A|^2
&=&\nabla f\cdot\nabla|A|^2+2|\nabla A|^2+(1-2|A|^2)|A|^2\\
&\geq& \Big(-\frac12|\nabla f|^2+1-2|A|^2\Big)|A|^2\\
&\geq& -\Big(\frac{|x|^2}8+2|A|^2\Big)|A|^2.
\eeqs
If we set $\varphi:=\frac{|x|^2}8+2|A|^2$, then 
\beqs
-\Delta|A|^2\leq \varphi|A|^2.
\eeqs
Recall that $\sup_{\Sigma}|H|\leq c_n K^{\frac12}$. Fix $x_0\in \Sigma$. Applying the standard Moser iteration Lemma \ref{lem:Moser iteration} yields that for any $\beta=\frac{n}2<q$
\beqs
\sup_{B(x_0,1)\cap\Sigma}|A|^2
&\leq& C(n,q)\Big(\|\varphi\|_{L^{q}(B(x_0,1)\cap\Sigma)}^{\frac{2q}{2q-n}}
+\|H\|_{L^{n+2}(B(x_0,1)\cap\Sigma)}^{n+2}\Big)^{2n}
\||A|^2\|_{L^{\frac{n}2}(B(x_0,1)\cap\Sigma)},
\eeqs
where by Lemma \ref{lem:vol est} and Proposition \ref{prop:Lp est}
\beqs
\|\varphi\|_{L^{q}(B(x_0,1)\cap\Sigma)}
&\leq& \frac18\||x|^2\|_{L^{q}(B(x_0,1)\cap\Sigma)}+2\||A|^2\|_{L^{q}(B(x_0,1)\cap\Sigma)}\\
&\leq& C(|x_0|^2+1)^{1+\frac{n}{2q}}+C(|x_0|^2+1)^{\frac{a'}{2q}},
\eeqs
\beqs
\|H\|_{L^{n+2}(B(x_0,1)\cap\Sigma)}^{n+2}\leq CK^{\frac{n+2}2}(|x_0|^2+1)^{\frac{n}2},
\eeqs
\beqs
\||A|^2\|_{L^{\frac{n}2}(B(x_0,1)\cap\Sigma)}\leq C(|x_0|^2+1)^{\frac{2a''}{n}}.
\eeqs
Finally we conclude for any $x\in\Sigma$,
\beqs
|A|(x)\leq C(|x|+1)^a,
\eeqs
for constants
$$a=a(n,q,K),$$
$$C=C(n,q,K,\int_{\Sigma}e^{-f},C_1(2q),C_1(n))=C(n,q,K,\int_{\Sigma}e^{-f},\int_{B(0,r_0)\cap\Sigma}|A|^{2q}),$$
where $r_0=c(n,q)(1+K)$.
So is Theorem \ref{theo:growth rate} proved.
\end{proof}

\section{Gap and compactness theorems}

Since Lemma \ref{lem:vol est} and Theorem \ref{theo:growth rate} show the polynomial growth, now we can consider global integrations with the natural weight $e^{-f}$, which leads to simplier calculations. As in the previous section, we set
\beqs
\sup_{\Sigma}|HA|\leq K.
\eeqs
We will use the notation $\Delta_{f}=\Delta-\nabla f\cdot\nabla$ which is self adjoint with respect to the weighted volume $e^{-f}dv$.
\begin{theo}[\textbf{gap theorem}]\label{theo:gap}
Let $\x:\Sigma^n\to\mathbb{R}^{n+1}$ be a properly immersed self-shrinker. If $\sup_{\Sigma}|HA|\leq\frac1{\sqrt{n}(n+5)^4}$, then $A\equiv0$.
\end{theo}
\begin{proof}
By virtue of (\ref{eqn2}) and (\ref{eqn3}), we have
\beqn\label{est1}
&&\int_{\Sigma}(f-\frac{n}2)|A|^p e^{-f}
=\int_{\Sigma}\Big(|\nabla f|^2-\Delta f\Big)|A|^p e^{-f}\\
&=&\int_{\Sigma}\Delta (e^{-f})|A|^p
= -\int_{\Sigma}\nabla(e^{-f})\cdot\nabla|A|^p\nonumber=\int_{\Sigma}\nabla f\cdot\nabla|A|^p e^{-f}.\nonumber
\eeqn
Note that $\sup_{\Sigma}|HA|\leq K$.
Using (\ref{eqn4}) and integrating by parts, we have
\beqs
&&\int_{\Sigma}\nabla f\cdot\nabla|A|^p e^{-f}
=p\int_{\Sigma}\nabla f\cdot\nabla A\cdot A |A|^{p-2} e^{-f}\\
&=&p\int_{\Sigma}\Big(\nabla^2 H+HA^2-\frac12 A\Big)A|A|^{p-2}e^{-f}\\
&\leq& p\int_{\Sigma}\nabla^2 H\cdot A |A|^{p-2}e^{-f}+ (pK-\frac{p}2)\int_{\Sigma}|A|^pe^{-f}\\
&\leq& 
p(p-1)\int_{\Sigma}|\nabla H||\nabla A||A|^{p-2}e^{-f}
+p\int_{\Sigma}|\nabla H||\nabla f||A|^{p-1}e^{-f}\\
&&+(pK-\frac{p}2)\int_{\Sigma}|A|^pe^{-f}.
\eeqs

By (\ref{eqn3}) and Schwarz's inequality,
\beqn
&&\int_{\Sigma}\nabla f\cdot\nabla|A|^p e^{-f}\label{est2}\\
&\leq& p(p-1)\int_{\Sigma}|\nabla H||\nabla A||A|^{p-2}e^{-f}
+\frac12\int_{\Sigma}|\nabla f|^2|A|^pe^{-f}\nonumber\\
&&+\frac{p^2}2\int_{\Sigma}|\nabla H|^2|A|^{p-2}e^{-f}
+(pK-\frac{p}2)\int_{\Sigma}|A|^pe^{-f}\nonumber\\
&\leq& p^2(1+\frac{\sqrt{n}}2)\int_{\Sigma}|\nabla H||\nabla A||A|^{p-2}e^{-f}+\frac12\int_{\Sigma}|\nabla f|^2|A|^pe^{-f}\nonumber\\
&&+(pK-\frac{p}2)\int_{\Sigma}|A|^pe^{-f}\nonumber\\
&\leq& p^2(1+\frac{\sqrt{n}}2)K\int_{\Sigma}|\nabla A|^2|A|^{p-4}e^{-f}+\frac14 p^2(1+\frac{\sqrt{n}}2)K^{-1}\int_{\Sigma}|\nabla H|^2|A|^pe^{-f}\nonumber\\
&&+\frac12\int_{\Sigma}f|A|^pe^{-f}+(pK-\frac{p}2)\int_{\Sigma}|A|^pe^{-f}.\nonumber
\eeqn
Plugging (\ref{est2}) into (\ref{est1}) yields 
\beqn
\frac12\int_{\Sigma}f|A|^pe^{-f}
&\leq& (pK+\frac{n}2-\frac{p}2)\int_{\Sigma}|A|^pe^{-f}
+p^2(1+\frac{\sqrt{n}}2)K\int_{\Sigma}|\nabla A|^2|A|^{p-4}e^{-f}\label{est3}\\
&&+p^2(1+\frac{\sqrt{n}}2)K^{-1}\int_{\Sigma}|\nabla H|^2|A|^pe^{-f}.\nonumber
\eeqn
Furthermore, by (\ref{eqn6}) we get, for $p\geq4$,
\beqn
&&\int_{\Sigma}|\nabla A|^2|A|^{p-4}e^{-f}
=\int_{\Sigma}\Big(\frac12\Delta_{f}|A|^2-(\frac12-|A|^2)|A|^2\Big)|A|^{p-4}e^{-f}\nonumber\\
&\leq& -\frac12\int_{\Sigma}\nabla|A|^2\cdot\nabla|A|^{p-4}e^{-f}+\int_{\Sigma}(|A|^p-\frac12|A|^{p-2})e^{-f}\leq \int_{\Sigma}|A|^pe^{-f},\nonumber
\eeqn
i.e.,
\beqn
\int_{\Sigma}|\nabla A|^2|A|^{p-4}e^{-f}\leq \int_{\Sigma}|A|^pe^{-f}.\label{est4}
\eeqn
Similarly, by (\ref{eqn8}) we get
\beqs
&&\int_{\Sigma}|\nabla H|^2|A|^pe^{-f}
=\int_{\Sigma}\Big(\frac12\Delta_{f}H^2+(|A|^2-\frac12)H^2\Big)|A|^pe^{-f}\\
&\leq&-\frac12\int_{\Sigma}\nabla H^2\cdot\nabla|A|^pe^{-f}+\int_{\Sigma}(H^2|A|^{p+2}-\frac12H^2|A|^p)e^{-f}\\
&\leq& p\int_{\Sigma}|\nabla H||\nabla A||H||A|^{p-1}e^{-f}+K^2\int_{\Sigma}|A|^pe^{-f}-\frac12\int_{\Sigma}H^2|A|^pe^{-f}\\
&\leq& \frac12\int_{\Sigma}|\nabla H|^2|A|^pe^{-f}+\frac{p^2}2\int_{\Sigma}|\nabla A|^2H^2|A|^{p-2}e^{-f}+K^2\int_{\Sigma}|A|^pe^{-f}\\
&\leq& \frac12\int_{\Sigma}|\nabla H|^2|A|^pe^{-f}+\frac{p^2}2 K^2\int_{\Sigma}|\nabla A|^2|A|^{p-4}e^{-f}+K^2\int_{\Sigma}|A|^pe^{-f},
\eeqs
which together with (\ref{est4}) implies
\beqn
\int_{\Sigma}|\nabla H|^2|A|^pe^{-f}\leq (p^2+2)K^2\int_{\Sigma}|A|^pe^{-f}.\label{est5}
\eeqn
Finally, combining (\ref{est3}), (\ref{est4}) and (\ref{est5}), we conclude
\beqn
\frac12\int_{\Sigma}f|A|^pe^{-f}\leq (cK+\frac{n}2-\frac{p}2)\int_{\Sigma}|A|^pe^{-f},\label{est12}
\eeqn
where 
$$c=p+p^2(p^2+3)(1+\frac{\sqrt{n}}2).$$
Now we take $p=n+4$. Then $c\leq 2\sqrt{n}(n+5)^4$. If $K\leq \frac1{\sqrt{n}(n+5)^4}$, i.e., a upper bound which depends only on $n$, then the above inequality implies that $A\equiv0$.
\end{proof}

\begin{cor}\label{cor:hyperplane gap}
    Let $\x:\Sigma^n\to\mathbb{R}^{n+1}$ be a smooth properly embedded self-shrinker. There exists a constant $\varepsilon_n=\frac1{\sqrt{n}(n+5)^4}$ such that if $\sup_{\Sigma}|HA|\leq\varepsilon_n$ then $\Sigma$ is a hyperplane through 0.
\end{cor}
\begin{proof}
Combining Theorem \ref{theo:gap} and Lemma \ref{lem:hyperplane} we immediately obtain Corollary \ref{cor:hyperplane gap}. 
\end{proof}

From the proof of Theorem \ref{theo:gap} we derive the following energy estimate.
\begin{prop}\label{prop:energy bound}
Let $\x:\Sigma^n\to\mathbb{R}^{n+1}$ be a properly immersed self-shrinker with $n\geq4$ and $\sup_{\Sigma}|HA|\leq K$. Then there exists a $r_1=c_n\sqrt{K}$ such that if 
$$\int_{B(0,r_1)\cap\Sigma}|A|^n\leq E,$$
then for any $r>0$ we have
$$\int_{B(0,r)\cap\Sigma}|A|^n\leq 3Ee^{r^2/4}.$$
\end{prop}
\begin{proof}
    Set $b=cK+\frac{n}2-\frac{p}2$ as long as it is positive.
    Dividing the integration in (\ref{est12}) into two parts, we see
    \beqs
    &&\int_{\{f\leq 3b\}\cap\Sigma}f|A|^p e^{-f}
    +3b\int_{\{f\geq 3b\}\cap\Sigma}|A|^p e^{-f}\\
    &\leq& 2b\int_{\{f\leq 3b\}\cap\Sigma}|A|^p e^{-f}
    +2b\int_{\{f\geq 3b\}\cap\Sigma}|A|^p e^{-f},
    \eeqs
    which implies
    \beqs
    \int_{\{f\geq 3b\}\cap\Sigma}|A|^p e^{-f}
    \leq 2\int_{\{f\leq 3b\}\cap\Sigma}|A|^p.
    \eeqs
    Moreover, for any $r>0$, 
    \beqs
    &&\int_{\{|x|\leq 2r\}\cap\Sigma}|A|^p
    =\int_{\{f\leq r^2\}\cap\Sigma}|A|^p\\
    &\leq& e^{r^2}\int_{\{f\geq 3b\}\cap\Sigma}|A|^p e^{-f}
    +\int_{\{f\leq 3b\}\cap\Sigma}|A|^p\\
    &\leq& (2e^{r^2}+1)\int_{\{f\leq 3b\}\cap\Sigma}|A|^p,
    \eeqs
    i.e.,
    \beqs
    \int_{\{|x|\leq r\}\cap\Sigma}|A|^p
    &\leq& (2e^{r^2/4}+1)\int_{\{|x|\leq 2\sqrt{3b}\}\cap\Sigma}|A|^p\\
    &\leq& 3e^{r^2/4}\int_{\{|x|\leq 2\sqrt{3cK+3(n-p)/2}\}\cap\Sigma}|A|^p.
    \eeqs
    Letting $p=n$ yields the energy estimate we want.
\end{proof}

Combining the volume estimate and the energy bound, we derive the following compactness theorem for self-shrikers, which largely follows the techniques on minimal surfaces. Remark that here we only need a local energy bound instead of a global bound.

\begin{theo}[\textbf{Compactness}]\label{theo:compactness}
Let $\{\Sigma^n_i\}$ be a sequence of properly embedded self-shrinkers with $n\geq4$ normalized by $\int_{\Sigma_i}e^{-f}\leq (4\pi)^{n/2}$. Assume that $\sup_{i}\sup_{\Sigma_i}|HA|\leq K$ and $\sup_{i}\int_{B(0,r_1)\cap\Sigma_i}|A|^n<\infty$ where $r_1=c_n\sqrt{K}$ is the positive constant in Proposition \ref{prop:energy bound}. Then a subsequence of $\{\Sigma_i\}$ converges smoothly to a smooth properly embedded self-shrinker $\Sigma_{\infty}$.
\end{theo}
\begin{proof}
Fix $B(0,\rho)\subset\mathbb{R}^{n+1}$. Assume that
$\int_{B(0,r_1)\cap\Sigma_i}|A|^n\leq E<\infty$. By Proposition \ref{prop:energy bound} we can define the measures $\nu_i$ on $B(0,\rho)$ by 
$$\nu_i(U):=\int_{U\cap B(0,\rho)\cap\Sigma_i}|A|^n\leq 3Ee^{\rho^2/4},\quad\forall\,U\subset B(0,\rho).$$
Then a subsequence converges weakly to a Radon measure $\nu$ with $\nu(B(0,\rho))\leq 3Ee^{\rho^2/4}$. We define the set 
$$S_{\rho}:=\{x\in B(0,\rho)\,|\,\nu(x)\geq \epsilon\},$$
where $\epsilon=\epsilon(n)$ is the positive constant in Lemma \ref{lem:epsilon-regularity} and see the number estimate
$\sharp\{S_{\rho}\}\leq \frac{3E}{\epsilon}e^{\rho^2/4}$. For any $x_0\in B(0,\rho)\setminus S_{\rho}$ there exists some $r\in(0,\frac{1}{n^{1/4}\sqrt{K}})$ such that $B(x_0,r)\subset B(0,\rho)\setminus S_{\rho}$ with $\nu(B(x_0,r))<\epsilon$. For $i$ sufficiently large we have
$$\int_{B(x_0,r)\cap\Sigma_i}|A|^n\leq\epsilon.$$
Applying Lemma \ref{lem:epsilon-regularity} yields interior estimates
$$\sup_{B(x_0,R/2)\cap\Sigma_i}|A|\leq \frac{D_k}{r},\quad\forall\,k\geq0,$$
where $R=R(r,n)<\frac{r}{32}$ and $D_k=D_k(n)$. By Lemma \ref{lem:compactness of MCF} and a diagonal sequence argument we find a subsequence of $\{B(0,\rho)\cap\Sigma_i\}$ converges smoothly, away from $S_{\rho}$, to a properly embedded self-shrinker $\Sigma_{\infty}$ in $B(0,\rho)$. Furthermore, we find a subsequence of $\{\Sigma_i\}$ converges in smooth topology, away from $S:=\bigcup_{\rho>0}S_{\rho}$, to a properly embedded self-shrinker $\Sigma_{\infty}$. By Lemma \ref{lem:vol est} and Lemma \ref{lem:volume ratio lower bound} the multiplicity of the convergence is bounded by $N_0=N_0(n,K)<\infty$. Note that $\Sigma_{\infty}$ is a minimal hypersurface with some conformal metric. Following the argument in the proof of Proposition 7.14 of \cite{[CM]} we see that $\Sigma_{\infty}\cup S$ is a smooth properly embedded self-shrinker and the convergence is also in Hausdorff distance.

Finally it boils down to the multiplicity-one convergence as in \cite{[CM1]}. Suppose for the sake of contradiction that $u_i$ denotes the normalized height-difference between the top and bottom sheets. In fact $\{u_i\}$ satisfies $Lu_i=0$ up to higher order correction terms and converges on $\Sigma_{\infty}\setminus S$ to a solution $u$ with $Lu=0$. Applying the foliation argument and local maximum principle in a cylindrical neighbourhood of a singularity $y\in S$, we see that $u_i$ is bounded on a neighbourhood of $y$ by a mutiple of its supermum on the boundary. Hence $u$ is bounded over each $y\in S$ and then extends to a smooth positive solution on $\Sigma_{\infty}\cup S$ which actually implies L-stability. However, there are no L-stable smooth properly embedded self-shrinkers according to Lemma \ref{lem:vol est} and Theorem 0.5 of \cite{[CM1]}. See more details in \cite{[CM1],[Sh]}.
\end{proof}


\begin{thebibliography}{10}
\bibitem{[CY]} B. L. Chen, L. Yin, \emph{ Uniqueness and pseudolocality theorems of the mean curvature flow}. Comm. Anal. Geom. 15 (2007), no. 3, 435-490.

\bibitem{[CW]} X. X. Chen, B. Wang, \emph{ On the conditions to extend Ricci flow(III)}. Int. Math. Res. Not. IMRN 2013, no. 10, 2349-2367.

\bibitem{[CZ]} X. Cheng, D. T. Zhou, \emph{ Volume estiamte about shrinkers}. Proc. Amer. Math. Soc. 141(2013), no. 2, 687-696. 

\bibitem{[Co]} A. Cooper, \emph{ A characterization of the singular time of the mean curvature flow}. Proc. Amer. Math. Soc. 139 (2011), no. 8, 2933-2942.

\bibitem{[CM]} T. H. Colding, W. P. Minicozzi II, \emph{ A course in minimal surfaces}. Graduate Studies in Mathematics, 121. American Mathematical Society, Providence, RI, 2011. xii+313 pp.

\bibitem{[CM1]} T. H. Colding, W. P. Minicozzi II, \emph{ Smooth compactness of self-shrinkers}. Comment. Math. Helv. 87 (2012), no. 2, 463-475.

\bibitem{[CM2]} T. H. Colding, W. P. Minicozzi II, \emph{ Generic mean curvature flow I: generic singularities}. Ann. of Math. (2) 175 (2012), no. 2, 755-833.

\bibitem{[Hu1]} G.Huisken, \emph{ Asymptotic behavior for singularities of the mean curvature flow}. J. Differential Geom. 31 (1990), no. 1, 285–299.

\bibitem{[Hu2]} G.Huisken, \emph{ Local and global behaviour of hypersurfaces moving by mean curvature}. partial differential equations on manifolds (Los Angeles, CA, 1990), 175–191,
Proc. Sympos. Pure Math., 54, Part 1, Amer. Math. Soc., Providence, RI, 1993.

\bibitem{[KMW]} B. Kotschwar, O. Munteanu, J. P. Wang, \emph{ A local curvature estimate for the Ricci flow}. J. Funct. Anal. 271 (2016), no. 9, 2604–2630.

\bibitem{[LS1]} N. Q. Le, N. Sesum, \emph{ The mean curvature at the first singular time of the mean curvature flow}. Ann. Inst. H. Poincare Anal. Non Lineaire 27 (2010), no. 6, 1441-1459.

\bibitem{[LS2]} N. Q. Le, N. Sesum, \emph{ On the extension of the mean curvature flow}. Math. Z. 267, 583–604 (2011).

\bibitem{[LS3]} N. Q. Le, N. Sesum, \emph{ Blow-up rate of the mean curvature during the mean curvature flow and a gap theorem for self-shrinkers}. Commun. Anal. Geom. 19(4), 633-659 (2011).

\bibitem{[Li]} Peter. Li, \emph{ Lecture notes on geometry analysis}. RIMGARC Lecture Notes Series 6. Seoul National University (1993).

\bibitem{[LW]} H. Z. Li, B. Wang, \emph{ The extension problem of the mean curvature flow(I)}, Invent. Math. 218, 721-777 (2019).

\bibitem{[LW2]} H. Z. Li, B. Wang, \emph{ On Ilmanen's multiplicity-one conjecture for mean curvature flow with type-I mean curvature}. arXiv:1811.08654.

\bibitem{[Ma]} C. Mantegazza, \emph{ Lecture notes on mean curvature flow}. Progress in Mathematics, 290. Birkh\"auser/Springer Basel AG, Basel, 2011. xii+166 pp.

\bibitem{[MW]} O. Munteanu, M. T. Wang, \emph{ The curvature of gradient Ricci solitons}. Math. Res. Lett. 18 (2011), no. 6, 1051–1069. (Reviewer: Bo Yang) 53C21 (53C25).

\bibitem{[Se]} N. Sesum, \emph{ Curvature tensor under the Ricci flow}. Amer. J. Math. 127 (2005), no. 6, 1315-1324.

\bibitem{[Sh]} B.Sharp, \emph{ Compactness of minimal hypersurfaces with bounded index}. J. Differential Geom. 106 (2017), no. 2, 317–339.

\bibitem{[W]} B. Wang, \emph{ On the conditions to extend Ricci flow(II)}. Int. Math. Res. Not. IMRN 2012, no. 14, 3192-3223.

\end{thebibliography}
\end{document}